\documentclass[12pt,reqno]{amsart}
\usepackage{amsfonts}

\usepackage{verbatim,amscd,amsmath,amsthm,amstext,amssymb,amsfonts,latexsym,lscape,rawfonts
}
\usepackage[arrow,matrix,curve]{xy}
\usepackage{xcolor}

\newtheorem{theorem}{Theorem}[section]

\newtheorem{lemma}[theorem]{Lemma}
\newtheorem{proposition}[theorem]{Proposition}
\theoremstyle{definition}
\newtheorem{definition}{Definition}

\newtheorem{remark}{Remark}
\newtheorem*{acknow}{Acknowledgments}

\newtheorem{example}{Example}

\newcommand{\R}{\mathbb{R}}
\newcommand{\C}{\mathbb{C}}
\newcommand{\D}{\mathbb{D}}

\newcommand{\St}{\mathbb{S}}
\newcommand{\F}{\mathbb{F}}
\newcommand{\Z}{\mathbb{Z}}

\begin{document}

\title[explicit Iwasawa for minimal Lagrangian surfaces]{Explicit expressions for the Iwasawa factors, the metric and the monodromy matrices for minimal Lagrangian surfaces in $\C P^2$}
\author{Josef F.~Dorfmeister}
\address{Fakult\"{a}t F\"{u}r Mathematik, TU-M\"{u}chen, Boltzmann Str. 3,
D-85747, Garching, Germany}
\email{dorfm@ma.tum.de}
\author{Hui Ma}
\address{Department of Mathematical Sciences, Tsinghua University,
Beijing 100084, P.R. China} \email{hma@math.tsinghua.edu.cn}
\thanks{The second author is partially supported by NSFC grant
No.~11271213.}

\maketitle

\begin{abstract} 
In this paper we continue our study of equivariant minimal Lagrangian surfaces in $\C P^2$, characterizing the rotationally equivariant cases and providing explicit formulae for relevant geometric quantities of translationally equivariant minimal Lagrangian surfaces in terms of Weierstrass elliptic functions.
\end{abstract}

\section{Introduction}

The study of minimal Lagrangian surfaces in the complex projective plane is an interesting subject from the point of view of differential geometry, mathematical physics and integrable systems theory. This paper is a continuation of \cite{DorfMa1} on the discussion of this subject via the loop group method.

The paper is organized as follows: in Section \ref{Sec:setup}, we recall the basic setup for minimal Lagrangian surfaces in $\C P^2$. In Section \ref{sec:vacuum},  we obtain that any vacuum can be deformed to the potential of the Clifford torus by an isometric  transformation and a coordinate change. In Section \ref{Sec:equiv}, we characterize the rotationally equivariant minimal Lagrangian surfaces in $\C P^2$. In Section \ref{Sec:trans}, we present the details of the computation for the Iwasawa decomposition of translationally equivariant minimal Lagrangian surfaces and we also give explicit solutions for the metrics and the associated family of immersions of such surfaces in terms of the Weierstrass $\wp-$functions. In Section \ref{Sec:quantities}, we provide explicit formulae for relevant geometric quantities of translationally equivariant minimal Lagrangian surfaces in terms of Weierstrass elliptic functions. 
In Section \ref{Sec:Hom}, we present a quite direct  classification of homogeneous minimal Lagrangian surfaces into $\C P^2$ by using the loop group method. 

\section{Basic setup of minimal Lagrangian surfaces in $\mathbb{C}P^2$}
\label{Sec:setup}

We recall briefly the basic set-up for minimal Lagrangian surfaces in $\mathbb{C}P^2$. For details we refer to \cite{MM, DorfMa1} and references therein.

Let $\mathbb CP^2$ be the complex projective plane endowed with the Fubini-Study metric  and $f:M\rightarrow {\mathbb C}P^2$ be a minimal Lagrangian immersion of an oriented surface. The induced metric on $M$ generates a conformal structure with
respect to which the metric is $g=2e^{u} dzd{\bar z}$,  
where $z=x+iy$ is a local conformal coordinate on $M$ and $u$ is a real-valued function defined on $M$ locally.
For any Lagrangian immersion $f$, there exists a local horizontal lift $F: {\rm U}\rightarrow S^5(1)$.  
We therefore have
\begin{equation*}\label{horizontal+conformal}
\begin{split}
& F_z \cdot {\overline F}=F_{\bar z}\cdot {\overline F} =0, \\
& F_z\cdot \overline{F_z}=F_{\bar z}\cdot \overline{F_{\bar z}}=e^{u},\quad 
 F_z\cdot \overline{F_{\bar z}}=0.
\end{split}
\end{equation*}
Thus 
$\mathcal{F}=(e^{-\frac{u}{2}}F_z, e^{-\frac{u}{2}}F_{\bar z}, F)$ is a Hermitian orthonormal moving frame 
globally defined on the universal cover of $M$.
Set
\begin{equation*}
\psi=F_{zz}\cdot\overline{F_{\bar z}}.
\end{equation*}
Then the cubic differential $\Psi=\psi dz^3$ is globally defined on $M$  and
independent of the choice of the local lift, which is called the Hopf differential of $f$.
One can obtain the Gauss-Codazzi equations of a minimal Lagrangian surface given by 
\begin{eqnarray}
u_{z\bar z}+e^u-e^{-2u}|\psi|^2&=0,\label{eq:mLsurfaces}\\
\psi_{\bar z}&=0. \label{eq:Codazzi}
\end{eqnarray}

We will need the following loop group decomposition.
\begin{theorem}[Iwasawa Decomposition theorem of $\Lambda SL(3,\mathbb{C})_{\sigma}$]\label{Thm:Iwasawa}
The multiplication map $\Lambda SU(3)_{\sigma}\times \Lambda^{+} SL(3, \C)_{\sigma} \rightarrow \Lambda SL(3, \C)_{\sigma}$ is 
surjective. Explicitly, 
every element $g\in \Lambda SL(3,\mathbb{C})_{\sigma}$ can be represented in the form
$g=h V_{+}$ with $h\in \Lambda SU(3)_{\sigma}$ and $V_{+}\in \Lambda^{+} SL(3,\mathbb{C})_{\sigma}$.
One can assume without loss of generality that $V_{+} (\lambda=0)$ has only positive diagonal entries.
In this case the decomposition is unique. 
\end{theorem}

\begin{example} For the Clifford torus $f: \mathbb{C}\rightarrow \mathbb{C}P^2$, we have a horizontal lift $F:\mathbb{C}\rightarrow S^5(1)$ as follows
$$F(z,\bar{z})=\frac{1}{\sqrt{3}} (e^{z-\bar{z}}, e^{\alpha z-\alpha^2 \bar{z}}, e^{\alpha^2 z-\alpha \bar{z}}),$$
where $\alpha=e^{\frac{2}{3}\pi i}$.
It is easy to see that $\psi=F_{zz}\cdot\overline{F_{\bar z}}=-1$ and $e^u=1$.
Then it follows from Wu's formula in \cite{DorfMa0}  that the normalized potential of the Clifford torus is given by
\begin{equation*}
\eta=\lambda^{-1}\begin{pmatrix}
0&0&i\\
i&0&0\\
0&i&0
\end{pmatrix} dz.
\end{equation*}

We write $\eta = \lambda^{-1} A dz$ and verify 
$[A,\tau(A)]=0$. Therefore the solution to $dC = C \eta, C(0,\lambda=1) =I$ is given by $C(z,\lambda) = \exp (z \lambda^{-1} A)$.
Moreover,  we can perform the Iwasawa decomposition directly  and obtain for the extended frame the expression
$\F(z,\lambda) = \exp( z \lambda^{-1} A + \bar{z} \lambda \tau(A))$. 
Consider the translation 
$$z\mapsto z+\delta, \quad\quad \text{with } \delta\in \C.$$
As a consequence,  the monodromy matrix of the frame $F(z,\lambda)$ for this translation is given by
$$\F(z+\delta,\lambda)=M(\delta, \lambda)\F(z,\lambda),$$
where $$M(\delta,\lambda)=\exp(\delta\lambda^{-1}A+\bar{\delta} \lambda \tau(A)).$$
As a consequence we obtain $F(z+\delta,\lambda)=M(\delta, \lambda)F(z,\lambda)$ and $f(z+\delta,\lambda)=M(\delta, \lambda)f(z,\lambda).$

Clearly, the map $f_{\lambda_0}: \C\rightarrow \C P^2$ can be defined on $\C/\delta \Z$ if and only if $f_{\lambda_0}(z+\delta)=f_{\lambda_0}(z)$.  By the above this is equivalent with 
$M(\delta, \lambda_0)f(z,\lambda_0) = f_{\lambda_0}(z)$ for all $z$. 
If we assume that $f$ is \lq\lq full\rq\rq and that it descends to a torus, then the last relation implies that $M(\delta, \lambda_0)$ is a multiple of identity, 
$M(\delta,\lambda_0)=cI$, where $c$ is  a scalar. Clearly then, $c$ needs to satisfy $c^3=1$.

Since the eigenvalues of $A$ are $i$, $i\alpha$ and $i\alpha^2$, it follows that the closing conditions for $\lambda_0\in S^1$ are
$$e^{i\lambda_0^{-1}\delta+i\lambda_0 \bar{\delta}}=e^{i\lambda_0^{-1}\alpha\delta+i\lambda_0 \alpha^2\bar{\delta}}
=e^{i\lambda_0^{-1}\alpha^2 \delta+i\lambda_0 \alpha \bar{\delta}}=c,$$
which is
\begin{eqnarray}
\mathrm{Re}(\lambda_0^{-1}\delta)&=&\frac{\pi+k\pi}{3}+l_1\pi, \label{periodcon1}\\
\mathrm{Re}(\lambda_0^{-1}\alpha\delta)&=&\frac{\pi+k\pi}{3}+l_2\pi, \label{periodcon2}\\
\mathrm{Re}(\lambda_0^{-1}\alpha^2\delta)&=&\frac{\pi+k\pi}{3}+l_3\pi, \label{periodcon3}
\end{eqnarray}
for $k=0,1$ or $2$ and $l_1,l_2,l_3\in \Z$. Then it is easy to see that for any $\lambda_0 \in S^1$, 
the solutions to \eqref{periodcon1}-\eqref{periodcon3} are given by
$$\delta=\frac{2l_1-l_2-l_3}{3}\lambda_0 \pi +i\frac{l_3-l_2}{\sqrt{3}}\lambda_0 \pi.$$
where $l_1+l_2+l_3+1+k=0$ for $k=0,1$ or $2$ and $l_1,l_2,l_3\in\Z$. 
Therefore, for arbitrary $\lambda_0$, we obtain $\delta(\lambda_0) \mathbb{Z} = \lambda_0 \delta(1) \mathbb{Z} $, i.e. the lattice $\delta(\lambda_0) \mathbb{Z} $ is obtained from the lattice $\delta(1) \mathbb{Z}$  by rotation by $\lambda_0$.
This implies the following
\begin{proposition} 
Every member in the associated family of the Clifford torus is  a torus.
\end{proposition}

\end{example}

\section{Vacuum solutions}
\label{sec:vacuum}

A \lq\lq vacuum\rq\rq is an extended framing whose normalized potential is given by $\eta=\lambda^{-1}A dz$ with $A\in \mathcal{G}_{-1}$ a constant matrix satisfying  $[A,\tau(A)]=0$, where $\tau$ is the conjugation of $SL(3, \C)$ with respect to the real form $SU(3)$ (see \cite{BuP}). To clarify what this means we consider the  constant matrix
$$A=\begin{pmatrix}
0&0&a\\
b&0&0\\
0&a&0
\end{pmatrix}\in \mathcal{G}_{-1}. 
\quad \text{ Then }
\tau(A)=\begin{pmatrix}
0&-\bar{b}&0\\
0&0&-\bar{a}\\
-\bar{a}&0&0
\end{pmatrix},$$
and the condition $[A,\tau(A)]=0$ says $|a|^2=|b|^2$.

Let's next write  $a=i re^{i\theta}$ and  $b=ir e^{i\beta}$.
Now take the following isometric transformation
$$\begin{pmatrix}
e^{i\delta}&&\\
&e^{-i\delta}&\\
&&1
\end{pmatrix}A
\begin{pmatrix}
e^{-i\delta}&&\\
&e^{i\delta}&\\
&&1
\end{pmatrix}=\begin{pmatrix}
0&0&ire^{i(\theta+\delta)}\\
ire^{i(\beta-2\delta)}&0&0\\
0&ire^{i(\theta+\delta)}&0
\end{pmatrix}.$$
Then choose $\delta$ such that $\theta+\delta=\beta-2\delta$,
i.e., $\delta=\frac{\beta-\theta}{3}$.
Thus, $$\eta=\lambda^{-1}ire^{i\frac{2\theta+\beta}{3}}\begin{pmatrix}0&0&i\\
i&0&0\\
0&i&0\end{pmatrix}dz.$$
Finally, choose a new coordinate: $z\mapsto w = r e^{i\frac{2\theta+\beta}{3}}z$, and we obtain $$\eta=\lambda^{-1}\begin{pmatrix}0&0&i\\
i&0&0\\
0&i&0\end{pmatrix}dw.$$

 Summing up we have
\begin{proposition}
Any vacuum can be deformed  by an isometric  transformation and a coordinate change (if necessary) to the potential of the Clifford torus.
\end{proposition}

\section{Equivariant minimal Lagrangian immersions into $\C P^2$}
\label{Sec:equiv}

\subsection{General background}

For all classes of surfaces, the surfaces admitting some symmetries are
of particular interest and beauty.

While a basic definition of a symmetry $R$ for a surface $f(M)$ may only mean $R f(M) = f(M)$, it is very useful to know that if the induced metric is complete, then on the universal cover $\tilde{M}$ of $M$ one finds some automorphism $\gamma$ such that 
\begin{equation} \label{basic-symmetry}
f(\gamma\cdot z) = R f(z) \hspace{2mm} \mbox{for all} \hspace{2mm} z 
\in \tilde{M}.
\end{equation}
Therefore, in this paper, a \lq\lq symmetry\rq\rq will always be a pair $(\gamma, R) \in
(Aut(M), Iso(\C P^2))$, such that (\ref{basic-symmetry}) holds.

The usual transition to the associated  family $f_{\lambda}$ then produces some family  $R(\lambda)$  of isometries of $\C P^2$ such that we have
\begin{equation*} \label{basic-lambda- symmetry}
f_{\lambda} (\gamma\cdot z) = R(\lambda) f_\lambda(z) \hspace{2mm} \mbox{for all} \hspace{2mm} z \in \tilde{M}.
\end{equation*}
More details can be found in \cite{DoHa;sym1}, \cite{DoHa;sym2}, \cite{DoWasym1}.

In this paper we will investigate minimal Lagrangian immersions for which there exists a one-parameter family  $(\gamma_t, R_t) \in (Aut(M), Iso(\C P^2))$ of symmetries.

\begin{definition}
Let $M$ be any connected Riemann surface and  $f:M \rightarrow \mathbb{C}P^2$ an immersion. Then $f$  is called  equivariant, relative to the one-parameter group $(\gamma_t, R(t)) \in (Aut(M), Iso(\C P^2))$, if 
$$f(\gamma_t \cdot p)=R(t) f(p)$$ for all $p\in M$ and all $t\in \mathbb{R}$.
\end{definition}

By the definition above, any Riemann surface $M$ admitting an equivariant minimal Lagrangian immersion admits a one-parameter group of (biholomorphic) automorphisms.
Fortunately, the classification of such surfaces is very simple:
\begin{theorem} (Classification of Riemann surfaces admitting one-parameter groups of automorphisms, e.g. \cite{Farkas-Kra})

\begin{enumerate}
\item  $S^2$,
\item  $\C$, $\mathbf{D}$,
\item  $\C^*$,
\item  $\mathbf{D}^*, \mathbf{D}_r$,
\item $T=\C/\Lambda_{\tau}$,
\end{enumerate}
where the superscript "$^*$" denotes deletion of the point $0$, the subscript "$r$" denotes the open annulus between $0 < r < 1/r$ and 
$\Lambda_{\tau}$ is the free group generated by the two translations $z\mapsto z+1$, $z\mapsto z+\tau$, $\mathrm{Im}\tau >0$.
\end{theorem}

Looking at this classification a bit more closely, one sees that after some biholomorphic transformations one obtains the following picture, including 
representative one-parameter groups:

\begin{theorem} \label{class-equi}
(Classification of Riemann surfaces admitting one-parameter groups of automorphisms and
representatives for the one-parameter  groups,e.g. \cite{Farkas-Kra})

(1) $S^2$, group of all rotations about the $z$-axis,

(2a) $\C$, group of all real translations,

(2b) $\C$, group of all rotations about the origin $0$,

(2c) $\mathbf{D}$, group of all rotations about the origin $0$,

(2d) $\mathbf{D} \cong \mathbb{H}$, group of all real translations,

(2e) $ \mathbf{D} \cong \mathbb{H} \cong \log  \mathbb{H} =\St$, the strip between $y=0$ and $y = \pi$, 
group of all real translations,

(3) $\C^*$, group of all rotations about $0$,

(4) $\mathbf{D}^*, \mathbf{D}_r,$ group of all rotations about $0$,

(5) $T$, group of all real translations.

\end{theorem}

For later purposes we state the following 

\begin{definition}
Let $f: M \rightarrow \C P^2$ be an equivariant minimal Lagrangian immersion, then $f$ will be called \lq\lq translationally equivariant\rq\rq, if the group of automorphisms acts by (all real) translations. It will be called \lq\lq rotationally equivariant\rq\rq, if the group acts by (all) rotations
(about the origin).
\end{definition}

\begin{remark} 
The case of  $S^2$ is usually special and has been treated in the literature. For minimal Lagrangian immersions this case has been treated in (\cite{Yau1974}) and it has been shown that  any minimal Lagrangian immersion $f$ from a sphere to $\C P^2$ is totally geodesic and it is the standard immersion of $S^2$ into $\mathbb{C}P^2$ (\cite{Yau1974}). Therefore, up to a few exceptions, we will exclude the case $S^2$ from the discussions in this paper. We would like to point out, however, that this case could be discussed like the general case below. In this case we would need to deal with algebraic solutions to elliptic equations listed below.
\end{remark}

\subsection{Rotationally equivariant minimal Lagrangian immersions}

In view of  Theorem \ref{class-equi} there are two types of equivariant surfaces, translationally equivariant surfaces and rotationally equivariant surfaces. The translationally equivariant case will be discussed in Section \ref{Sec:trans}. Thus here it remains to consider
rotationally equivariant minimal Lagrangian immersions.

There are essentially three types of such surfaces: those without fixed point in $M$, 
those with exactly one fixed point in $M$ and those with two fixed points in $M$, i.e. $M =S^2$. Let's first consider the cases $\C^*, \mathbf{D}^*, \mathbf{D}_r$, which do not contain the fixed point of the group of rotations.

Using the covering map $ w \rightarrow exp(iw)$, we see that the rotationally symmetric minimal Lagrangian immersion  $f$ is obtained from some translationally equivariant minimal Lagrangian immersion $\tilde{f}$ defined on some strip $\St$. Obviously, the condition of descending to the given rotationally equivariant minimal Lagrangian immersion is equivalent with
$\tilde{f}$ being $ 2 \pi-$periodic in the variable corresponding to the 
group of translations. Since $\tilde{f}$ is actually defined on $\C$ and real analytic, it is clear that $\tilde{f}$ is $2 \pi-$periodic on $\C$ and thus descends to a rotationally equivariant minimal Lagrangian immersion on $\C^*$.

\begin{theorem}
Consider a rotationally equivariant minimal Lagrangian immersion $f:M \rightarrow \C P^2$ defined 
on $M =\C^*, \mathbf{D}^*, \mathbf{D}_r$.
Then $f$ can be extended without loss of generality to $\C^*$ and can be obtained from some 
$ 2 \pi-$periodic translationally equivariant minimal Lagrangian immersion defined on $\C$ by projection.
\end{theorem}

Next we consider the rotationally equivariant minimal Lagrangian immersions defined on $\C$ or $\mathbf{D}$. In these cases we remove the fixed point $0$ and obtain  rotationally equivariant minimal Lagrangian immersions without fixed point. The last theorem shows that we only need to consider the case $M = \C$. Clearly this is a special case of a rotationally equivariant minimal Lagrangian immersion defined on $\C^*$. Finally,  considering $S^2$, we can assume without loss of generality that the group acts by rotations about the $z-$axis. Then any rotationally equivariant minimal Lagrangian immersion on $S^2$ is a special case of a rotationally equivariant minimal Lagrangian immersion defined on $\C$.

\subsection{Rotationally equivariant minimal Lagrangian immersions defined on $\C$}
\label{subsect:Rotationally equiv}

Let $f:\C \rightarrow \C P^2$ be a rotationally equivariant minimal Lagrangian immersion and $\F(z,\lambda)$ an extended frame. 
We normalize $\F$ by $\F(z=0,\lambda)=I$.
Then the equivariance is reflected by the equation
\begin{equation*}
\F(e^{it}z, \lambda) = \chi(e^{it}, \lambda) \F(z,\lambda) \mathcal{K}(e^{it}, z).
\end{equation*}

Setting $z=0$ shows $\chi(e^{it}, \lambda) \mathcal{K}(e^{it}, 0) = I.$
As a consequence, $\chi$ is independent of $\lambda$ and a one-parameter group in $K$. Hence 
\begin{equation*}
\chi(e^{it}, \lambda) = \exp( it r\delta),
\end{equation*}
where $r \in \mathbb{R}$, $t \in \R$ and $\delta = \mbox{diag}(1,-1,0)$.

Performing a Birkhoff splitting  of $\F$, $\F = \F_- V_+,$ we derive
\begin{equation*}
\F_- ( e^{it}z, \lambda) =  
\exp( itr \delta) \F_- ( z, \lambda) \exp( -itr\delta). 
\end{equation*}

For the Maurer-Cartan form $\eta_- = {\F_-}^{-1} d \F_- $ of 
$\F_- $ we then obtain
\begin{equation*}
(e^{it})^* \eta_- = \exp( itr\delta) \eta_- \exp(- itr\delta).
\end{equation*}
Hence the normalized potential $\eta_-$ has the form $\eta_- = \lambda^{-1} A dz,$ where

\begin{equation*}
A=\begin{pmatrix}0&0&az^{m-1}\\
bz^{-2m-1}&0&0\\
0&az^{m-1}&0
\end{pmatrix}
\end{equation*}
with $m\in \mathbb{Z}$ and certain complex numbers $a$ and $b$.

Since we had normalized everything at $z=0$, the normalized potential is holomorphic at $z=0$. Hence $ m \geq 1$ and $b=0$, 
which implies that 
the cubic Hopf differential $\Psi$ vanishes. It follows from
\eqref{eq:mLsurfaces} that the Gauss curvature satisfies 
$K=-u_{z\bar{z}}e^{-u}=1$. 
Then from the Gauss equation we obtain that
$\frac{S}{2}=1-K$ vanishes,
where $S$ is the norm square of the second fundamental form of the surface. Therefore $f$ is totally geodesic, hence 
the image of this minimal Lagrangian immersion $f$ lies in $\R P^2$ up to isometries of $\C P^2$.
We thus have reproved part of Corollary 3.9 of \cite{EschenburgGT}.
As a consequence we obtain

\begin{theorem}
Any  minimal Lagrangian immersion $f$ from $\C$ or $S^2$ into $\C P^2$ which is rotationally equivariant has a vanishing cubic Hopf differential, and therefore  is  totally  geodesic in $\C P^2$ and  its image
 is, up to isometries of $\C P^2$, contained in 
$\R P^2$.
\end{theorem}

\begin{remark}
This result can also be obtained by using the explicit Iwasawa decomposition discussed below. Also see Remark \ref{rem:4}.
\end{remark}

\section{Explicit discussion of translationally equivariant minimal Lagrangian immersions}
\label{Sec:trans}

\subsection{Burstall-Kilian theory for translationally equivariant minimal Lagrangian immersions}
\label{subsect:BK potential}

We now consider  translationally equivariant minimal Lagrangian immersions defined on some strip $\St$ with values in $\C P^2$,
$f:\St \rightarrow \mathbb{C} P^2$, i.e. minimal Lagrangian immersions  for which there exists a one-parameter subgroup $R(t)$ of $SU(3)$ such that 
$$f(t + z, t+\bar{z})=R(t) f(z,\bar{z})$$
for all $z\in \St$. 
Following the approach of \cite{BuKi}, we have shown

\begin{theorem}[\cite{DorfMa1}] \label{essential-transfo}
For the extended frame $\F$ of any translationally equivariant minimal Lagrangian immersion we can assume without loss of generality 
$\F(0,\lambda)=I$ and
\begin{equation*}
\mathbb{F}(t+z, \lambda)=\chi(t,\lambda)\mathbb{F}(z,\lambda),
\end{equation*}
with $\chi(t,\lambda)=e^{tD(\lambda)}$ for some $D (\lambda)\in \Lambda su(3)_{\sigma}$. 
\end{theorem}

Note that  $\F$  satisfies 

\begin{equation*}\label{eq:mathbbF}
\begin{split}
\mathbb{F}^{-1}\mathbb{F}_z&=
\frac{1}{\lambda}\left(
   \begin{array}{ccc}
     0 & 0 & i e^{\frac{u}{2}} \\
  -i \psi e^{-u}   & 0 & 0 \\
     0 &  i e^{\frac{u}{2}} & 0 \\
   \end{array}
 \right)+\left(
   \begin{array}{ccc}
   \frac{u_z}{2}  & &  \\
      & -\frac{u_z}{2} &  \\
      & & 0 \\
   \end{array}
 \right)\\
 &:=\lambda^{-1}U_{-1}+U_0,\\
\mathbb{F}^{-1}\mathbb{F}_{\bar{z}} &=\lambda \left(
   \begin{array}{ccc}
     0 &  -i\bar{\psi} e^{-u} & 0 \\
     0& 0   & i e^{\frac{u}{2}} \\
     i e^{\frac{u}{2}} &0 & 0 \\
   \end{array}
 \right)+\left(
   \begin{array}{ccc}
     -\frac{u_{\bar z}}{2} & &  \\
      & \frac{u_{\bar z}}{2}  & \\
     & & 0 \\
   \end{array}
 \right)\\
  &:=\lambda V_{1}+V_0.
\end{split}
\end{equation*}
and we can assume 
\begin{equation} \label{equiF}
\F(x+iy, \lambda) = e^{(x+iy) D(\lambda)} U_+(y, \lambda)^{-1} = e^{xD(\lambda)} \F(iy,\lambda),
\end{equation}
with $U_+(y,\lambda) \in \Lambda^+ SL(3,\C)_\sigma.$

\subsection{The basic set-up for an explicit Iwasawa decomposition}

We have seen above in subsection \ref{subsect:BK potential} that every translationally equivariant minimal Lagrangian immersion can be obtained from some potential of the form 
\begin{equation*}
\eta = D(\lambda) dz,
\end{equation*}
where
\begin{equation*}
D(\lambda) = \lambda^{-1} D_{-1} + D_0 + \lambda D_1 \in \Lambda su(3)_\sigma.
\end{equation*}

The general loop group approach requires to consider the solution to $dC = C \eta, C(0,\lambda ) = I$. This is easily achieved by
$C(z,\lambda) = \exp(zD).$

Next one needs to perform an Iwasawa splitting.
In general this is very complicated and difficult to carry out explicitly.
But, for translationally equivariant minimal Lagrangian surfaces in $\mathbb{C}P^2$,  one is able to carry out an explicit Iwasawa decomposition of $\exp(zD)$.

In view of equation \eqref{equiF} we obtain 
\begin{equation}\label{eq:F(iy)}
\F (iy,\lambda)= e^{iyD} U_{+}(y,\lambda)^{-1}.
\end{equation}
Using  \eqref{eq:F(iy)} we obtain for the Maurer-Cartan form $\alpha = \F^{-1} d\F = Adx + B dy$ of $\F$ the equations
\begin{equation}\label{eq:A_lambda}
A_{\lambda}(y)=U_{+}(y,\lambda)DU_{+}(y,\lambda)^{-1},
\end{equation}
\begin{equation*} \label{eq:B_lambda}
B_{\lambda}(y)=U_{+}(y,\lambda)i DU_{+}(y,\lambda)^{-1}-\frac{d}{dy}U_{+}(y,\lambda) U_{+}(y,\lambda)^{-1}.
\end{equation*}

Writing, on the other hand, $\alpha = U + V$ with $U$ a $(1,0)-$form and $V$ a $(0,1)-$form, we obtain

\begin{equation} \label{eq:Equ1}
U_{+}(y,\lambda)DU_{+}^{-1}(y,\lambda)=
\lambda^{-1}U_{-1}+U_0+\lambda V_1+V_0
=:\Omega,
\end{equation}
\begin{equation} \label{eq:Equ2}
\frac{d}{dy}U_{+}(y,\lambda) U_{+}(y,\lambda)^{-1}=
2i(\lambda V_1+V_0).
\end{equation}
The above two equations are the basis for an explicit computation of the Iwasawa decomposition of 
$\exp(z D(\lambda))$.

It is important to note that
because $U_{+}$ only depends on $y$ and $\Omega$  
is of the form
$$\Omega
=\begin{pmatrix}
\frac{u_z-u_{\bar z}}{2}&-i\lambda \bar{\psi}e^{-u}&i\lambda^{-1}e^{\frac{u}{2}}\\
-i\lambda^{-1}\psi e^{-u}&-\frac{u_z-u_{\bar z}}{2}&i\lambda e^{\frac{u}{2}} \\
i\lambda e^{\frac{u}{2}}&i\lambda^{-1} e^{\frac{u}{2}}&0
\end{pmatrix},$$
both $u$ and $\psi$ also only depend on $y$. 

\begin{lemma}
If $f$ is a translationally equivariant minimal Lagrangian immersion  into $\C P^2$ with respect to  translations in $x$-direction, then the metric only depends on $y$ and the cubic Hopf differential has a constant coefficient.
\end{lemma}

There will be two steps for the computation of the Iwasawa decomposition of $\exp(z D(\lambda))$.

{\bf Step 1:} Solve  equation \eqref{eq:Equ1} in any way one pleases
by some matrix $Q$.
Then $U_+$ and $Q$ satisfy 
\begin{equation*}
U_+ = Q E, \hspace{2mm} \text{ where}  \hspace{2mm} E \hspace{2mm}  \text{commutes with} \hspace{2mm} D.
\end{equation*}

{\bf Step 2:} Solve  equation \eqref{eq:Equ2}. This will generally only mean to carry out two integrations in one variable.

\subsection{Evaluation of the characteristic polynomial equations}

Step 1 mentioned above actually consists of two sub-steps.
First of all one determines $\Omega$ from $D$ and then one computes a solution $W$ to {\bf Step 1:} Solve the equation \eqref{eq:Equ1}.

In this section we will discuss the first sub-step. 
In our case we observe that $D$ and $\Omega$ are conjugate and therefore have the same characteristic polynomials.
Using the explicit form of $\Omega$ stated just above and 
writing $D$ in the form 
\begin{eqnarray*}
D=\begin{pmatrix}
\alpha &-\lambda \bar{b} & \lambda^{-1} a \\
\lambda^{-1} b &-\alpha&-\lambda \bar{a}\\
-\lambda \bar{a} &\lambda^{-1} a &0
\end{pmatrix}\in \Lambda_1 \subset \Lambda su(3)_{\sigma},
\end{eqnarray*}
where $\alpha, a$ and $b$ are constants,
\eqref{eq:Equ1} is equivalent to
\begin{eqnarray}
&&2e^{u}+|\psi|^2e^{-2u}+\frac{1}{4}(u^\prime)^2=-\alpha^2+2|a|^2+|b|^2=:\beta,\label{eq:equiv1}\\
&&\psi=-ia^2b,\label{eq:equiv2}
\end{eqnarray}
where $\alpha, a, b$ and $\psi$ are constants. 

\begin{remark}
It has been noticed long time ago that the cases $\psi = 0$ and $u =$constant are related with very special surfaces. As pointed out above, the case $\psi=0$ implies that the surface is an open portion of $\R P^2$. The case $u =$ constant yields a flat surface thus is, up to isometries, an open portion of the Clifford torus (\cite{LOY}). Both cases can be treated like the general case below. In the first case, as already mentioned, one needs to use hyperbolic solutions
and in the second case one needs to use constant solutions
 to the elliptic equations occurring in this context.
At any rate, from here on (unless stated explicitly otherwise) we will assume that $\psi$ is not identically $0$ and $u$ is not constant. In particular, we will assume $u^{\prime}\not \equiv 0$.
\end{remark}

\subsection{Explicit solutions for metric and cubic form in terms of Weierstrass $\wp-$functions}

We start by noticing that \eqref{eq:equiv1} is a first integral of the Gauss equation
\begin{equation}\label{eq:GaussODE0}
\frac{1}{4}u^{\prime\prime}+e^u-|\psi|^2 e^{-2u}=0.
\end{equation}

Making the change of variables $w=e^{u}$ in \eqref{eq:equiv1}, we obtain equivalently
\begin{equation}\label{eq:w0}
(w^{\prime})^2+8w^3-4 \beta w^2 +4|\psi|^2=0.
\end{equation}
Set $w(y)=\frac{\beta}{6}-\frac{v(y)}{2}$. We obtain the fundamental differential equation
\begin{equation*}
(v^{\prime})^2=4v^3-g_2 v-g_3
\end{equation*}
of the Weierstrass function $\wp(z)=\wp(z; g_2,g_3)$ with
\begin{equation*}
g_2=\frac{4}{3}\beta^2, \quad\quad g_3=16|\phi|^2-\frac{8}{27}\beta^3.
\end{equation*}
Thus the general non-constant solution to \eqref{eq:w0} can be given by
\begin{equation*}
w(y)=\frac{\beta}{6}-\frac{\wp(y-y_0;g_2,g_3)}{2}
\end{equation*}
for a constant $y_0$.

 Since the Weierstrass elliptic function is periodic and bounded along the real line now, there exists a point,  where the derivative of $u$ vanishes. 
Choosing this point as the origin, we can always assume $u^{\prime}(0)=0$, which leads to
\begin{equation}\label{initial condition}
w^{\prime}(0)=0.
\end{equation}

This convention in combination with \eqref{eq:equiv1}  implies
\begin{equation*}\label{eq:equia1beta}
2a_1+\frac{|\psi|^2}{a_1^2}=\beta,
\end{equation*}
where $a_1:=e^{u(0)}>0$. 
Considering $\beta$ as a function of $a_1$, one can easily see that $\beta^3\geq 27|\psi|^2$. 
Thus the discriminant of the cubic equation 
\begin{equation}\label{W-equation}
4v^3-g_2 v-g_3=0
\end{equation}
satisfies
\begin{equation}\label{Delta-e}
\Delta=g_2^3-27 g_3^2=256 |\psi|^2 (\beta^3-27|\psi|^2)>0,
\end{equation}
if and only if $\psi \neq 0$ and $\beta^3\neq 27|\psi|^2$.

\begin{remark}\label{rem:4}
When $\psi=0$,  $g_2=\frac{4\beta^2}{3}>0$, $g_3=-\frac{8}{27}\beta^3<0$ and  the roots of the Weierstrass equation \eqref{W-equation} are 
$$e_1=e_2=\frac{\beta}{3}, \quad e_3=-\frac{2\beta}{3},$$
and the Weierstrass elliptic function $\wp(z; g_2, g_3)$ is reduced to (cf. 18.12.3, \cite{AS})
$$\wp(z; g_2,g_3)=\frac{\beta}{3}+\beta [\sinh (\sqrt{\beta}z)]^{-2}.$$
Thus the initial condition \eqref{initial condition} gives the following solution to \eqref{eq:w0}
$$w(y)=\frac{\beta}{2\cosh^2(\sqrt{\beta}y)}.$$
This is nothing but the metric of the real projective plane in $\mathbb{C}P^2$.

When $\beta^3=27|\psi|^2$, $g_2=\frac{4\beta^2}{3} >0$, $g_3=\frac{8}{27}\beta^3 >0$ and there also are two equal real roots of \eqref{W-equation} given by (cf. 18.12.25, \cite{AS})
$$e_1=\frac{2\beta}{3}, \quad e_2=e_3=-\frac{\beta}{3},$$
then we obtain the three roots of 
\begin{equation}\label{w-Wequation}
8w^3-4\beta w^2+4|\psi|^2=0
\end{equation}
given by
 $$a_1=a_2=\frac{\beta}{3}, \quad a_3=-\frac{\beta}{6}.$$
Thus the solution of \eqref{eq:w0} is a constant function $w(y)\equiv \frac{\beta}{3}$,
which corresponds to a flat minimal Lagrangian surface and has been ruled out in the beginning of our discussion.
\end{remark}

Now for the general case,  $g_2$ and $g_3$ are real and $\Delta>0$, thus there are three distinct non-zero real roots of  \eqref{W-equation}, denoted by
$$e_1>e_2>e_3.$$
Because of the initial condition \eqref{initial condition}, we know from \eqref{eq:w0}  that
$$e^{u(0)}=w(0)=:a_1$$
is a root of   \eqref{w-Wequation}. 
And now we assume that this is the largest root of  \eqref{w-Wequation}.
Recall that the half-periods $\omega$, $\omega^{\prime}$ and $\omega+\omega^{\prime}$ of the Weierstrass elliptic function are related to the roots
$e_1, e_2$ and $e_3$ by 
$$\wp(\omega)=e_1,\qquad \wp(\omega+\omega^{\prime})=e_2,\qquad \wp(\omega^{\prime})=e_3.$$
Consequently, $\wp^{\prime}(\omega)=\wp^{\prime}(\omega^{\prime})=\wp^{\prime}(\omega+\omega^{\prime})=0$.
The initial condition \eqref{initial condition} thus yields the particular solution of \eqref{eq:w0} given by
\begin{equation}\label{sol}
e^{u(y)}=w(y)=\frac{\beta}{6}-\frac{\wp(y-\omega^{\prime}; g_2,g_3)}{2}.
\end{equation}

Remark that now  
the half-period $$\omega=\int_{e_1}^{\infty}\frac{dt}{\sqrt{4t^3-g_2 t-g_3}}$$ is real, whereas the other half-period 
$$\omega^{\prime}
=i \int^{e_3}_{-\infty} \frac{dt}{\sqrt{|4t^3-g_2 t-g_3|}}$$ is purely imaginary.
It is easy to see that the solution $u(y)$ inherits
from the Weierstrass elliptic function the following properties:
\begin{itemize}
\item[(1)] $u(y+2\omega)=u(y), $
\item[(2)] $u(-y)=u(y)$,
\item[(3)] $u(\omega)=\log a_2$ and $u^{\prime}(\omega)=0$, where $a_2=\frac{\beta}{6}-\frac{e_2}{2}>0$.
\end{itemize}
In particular,  $\hat{u}(y)=u(y+\omega)$ is also a solution to \eqref{eq:GaussODE0} with 
$\hat{u}^{\prime}(0)=0$.

Thus  for any (in $x-$direction) translationally equivariant minimal Lagrangian surface in $\mathbb{C}P^2$, its metric conformal factor $e^u$ is given by  \eqref{sol} in terms of a Weierstrass elliptic function and its cubic Hopf differential is constant and given by \eqref{eq:equiv2}.

\bigskip

For our loop group setting the assumption $u^{\prime}(0)=0$ has an important consequence:

\begin{theorem}
By choosing the coordinates such that the metric for a given translationally equivariant minimal Lagrangian immersion has a vanishing derivative at $z = 0$, we obtain that the generating matrix $D$ satisfies $D_0 = 0$.
\end{theorem}

We will therefore always assume this condition from here on.

\subsection{Solving equation \eqref{eq:Equ1}}

The main goal of this subsection is to find some \lq\lq sufficiently nice\rq\rq matrix function $Q$ satisfying \eqref{eq:Equ1}, i.e.
\begin{equation}\label{eq:QDQ-1}
QDQ^{-1}=\Omega.
\end{equation}

Recall that for translationally equivariant minimal Lagrangian surfaces, the potential matrix $D$ coincides with $A_{\lambda}(0)=\Omega|_{y=0}$ of \eqref{eq:A_lambda} so we have 
(including the convention above about the origin)
\begin{equation*}\label{eq:D=Alambda0}
D=\begin{pmatrix}
0&-i\lambda\bar{\psi}e^{-u(0)}&i\lambda^{-1}e^{\frac{u(0)}{2}}\\
-i\lambda^{-1}\psi e^{-u(0)}&0&i\lambda e^{\frac{u(0)}{2}}\\
i\lambda e^{\frac{u(0)}{2}}&i\lambda^{-1} e^{\frac{u(0)}{2}}&0
\end{pmatrix},
\end{equation*}
where  $\alpha=-\frac{iu^{\prime}(0)}{2}=0$, $a=ie^{\frac{u(0)}{2}}$
and $b=-i\psi e^{-u(0)}$.
We may summarize the following proposition:
\begin{proposition} Up to isometries in $\mathbb{C}P^2$, any translationally equivariant minimal Lagrangian surface can be generated by a potential of the form 
\begin{equation}\label{eq:Dalpha=0}
\begin{pmatrix}
0&-\lambda \bar{b}&\lambda^{-1}a\\
\lambda^{-1}b&0&-\lambda \bar{a}\\
-\lambda\bar{a}&\lambda^{-1} a&0
\end{pmatrix} dz,
\end{equation}
where $a$ is purely imaginary and both $a$ and $b=\frac{i\psi}{a^2}$ are constants.
\end{proposition}

Thus the characteristic polynomial of $D$ in \eqref{eq:Dalpha=0} is given by
\begin{equation*} \label{eq:charpol}
\det(\mu I-D(\lambda))=\mu^3+\beta \mu-2i \mathrm{Re}(\lambda^{-3}\psi).
\end{equation*}

\begin{remark}\label{rem:six lambda}
It is easy to derive from \eqref{Delta-e} that the discriminant of the above polynomial satisfies
$$\Delta=(\frac{\beta}{3})^3-[\mathrm{Re}(\lambda^{-3}\psi)]^2\geq (\frac{\beta}{3})^3-|\psi|^2\geq 0.$$
The second equal sign holds when $\lambda^{-3}\psi$ is real and the third one holds only for special cases which we excluded.
Hence for the general case when $\Delta>0$, $D(\lambda)$ has three distinct purely imaginary roots for any choice of $\lambda^{-3}\psi$. Moreover, the root $0$ occurs if and only if $\lambda^{-3}\psi$ is purely imaginary. This case  can only happen for six different values of $\lambda$ (See Lemma 5.3 in \cite{DorfMa1}).  
\end{remark}

Denote the eigenvalues of $D(\lambda)$ by $\mu_1=i d_1$, $\mu_2=i d_2$, $\mu_3=i d_3$. 

The following relations will be frequently used later.
$$d_1+d_2+d_3=0, \quad d_1d_2+d_2d_3+d_3 d_1=-\beta, \quad d_1d_2d_3=-2\mathrm{Re}(\lambda^{-3}\psi).$$

Now take 
\begin{equation}\label{eq:Q_0}
Q_0=\mathrm{diag}(i a^{-1} e^{\frac{u}{2}}, -i a e^{-\frac{u}{2}},1),
\end{equation}
such that 
$$\hat{\Omega}=Q_0^{-1}\Omega Q_0=\begin{pmatrix}
-\frac{iu^{\prime}}{2}&i\lambda \bar{\psi}a^2 e^{-2u}&\lambda^{-1}a\\
\lambda^{-1}b&\frac{iu^{\prime}}{2}&-\lambda a^{-1}e^u\\
-\lambda a^{-1}e^u&\lambda^{-1}a&0
\end{pmatrix}$$
has the same coefficients at $\lambda^{-1}$ as $D(\lambda)$. 

Consider now  a $3\times 3$ matrix $\hat{Q}$ given by
\begin{equation}\label{eq:Q_assumption}
\hat{Q}=\begin{pmatrix}
A&\gamma\\
0&c
\end{pmatrix}, \quad A=\begin{pmatrix}p&q\\s&t\end{pmatrix},\quad \gamma=\begin{pmatrix} v_1\\v_2\end{pmatrix},
\end{equation}
where $c$ is a scalar. Put
$$D=\begin{pmatrix}
E&\xi\\
-\bar{\xi}^t&0
\end{pmatrix}, \quad 
\hat{\Omega}=\begin{pmatrix}
\Omega^{\prime}&\eta\\
\zeta&0
\end{pmatrix},$$
where $E$ and $\Omega^{\prime}$ are $2\times 2$ matrices, $\xi$ and $\eta$ are $2\times 1$ matrices, and $\zeta$ is a $1\times 2$ matrix.

Then $\hat{Q}D\hat{Q}^{-1}= \hat{\Omega}$ is equivalent to the following equations
\begin{eqnarray}
&&AEA^{-1}-\gamma \bar{\xi}^t A^{-1}=\Omega^{\prime}, \label{eq:Q1}\\
&&-(AEA^{-1}\gamma-\gamma\bar{\xi}^t A^{-1}\gamma)+A\xi=c\eta,
\label{eq:Q2}\\
&&-c\bar{\xi}^t A^{-1}=\zeta,
\label{eq:Q3}\\
&&\bar{\xi}^t A^{-1}\gamma=0. \label{eq:Q4}
\end{eqnarray}
Inserting \eqref{eq:Q3} into \eqref{eq:Q4} gives $\zeta \gamma=0$, which implies that
$$\gamma=\begin{pmatrix}\lambda^{-1}a^2\\\lambda e^u\end{pmatrix} h,$$
where $h$ is an arbitrary factor.
Noticing that \eqref{eq:Q1} is equivalent to $AE-\gamma\bar{\xi}^t=\Omega^{\prime} A$ and inserting the assumption \eqref{eq:Q_assumption}, we obtain the following equivalent equations
\begin{eqnarray}
&&\frac{iu^{\prime}}{2}p +\lambda^{-1} bq-i\lambda a^2 \bar{\psi} e^{-2u} s-|a|^2ah=0,\label{eq:Q6}\\
&&-\lambda\bar{b}p+\frac{i u^{\prime}}{2} s-i\lambda a^2 \bar{\psi}e^{-2u}t+\lambda^{-2} a^3 h=0,\label{eq:Q7}\\
&&-\lambda^{-1}bp-\frac{i u^{\prime}}{2}s+\lambda^{-1}bt-\lambda^2 \bar{a}e^u h=0,\label{eq:Q8}\\
&&-\lambda^{-1}bq-\lambda\bar{b}s-\frac{i u^{\prime}}{2} t+ae^u h=0.\label{eq:Q9}
\end{eqnarray}

Multiplying both sides of \eqref{eq:Q6} by $-\frac{iu^{\prime}}{2}$, \eqref{eq:Q7} by $\lambda^{-1}b$ and adding them together, we infer
\begin{equation*}
\begin{split}
&(\frac{(u^{\prime})^2}{4}-|b|^2)p+i\lambda a^2\bar{\psi}e^{-2u}\frac{iu^\prime}{2}s-ia^2b\bar{\psi}e^{-2u}t\\
&\quad\quad\quad \quad\quad\quad +(\lambda^{-3} a^3 b+|a|^2 a\frac{iu^\prime}{2})h=0.
\end{split}
\end{equation*}
Multiplying \eqref{eq:Q8} by $i\lambda a^2 \bar{\psi}e^{-2u}$ and adding the above equation to eliminate $s$ and $t$, we obtain
\begin{eqnarray*}
p&=&\frac{|a|^2a(-\frac{iu^{\prime}}{2}+i\lambda^3 \bar{\psi}e^{-u})-\lambda^{-3}a^3b}{\frac{(u^\prime)^2}{4}-|b|^2+|\psi|^2 e^{-2u}}h.\\
\end{eqnarray*}
Similarly, multiplying both sides of \eqref{eq:Q6} by $\lambda \bar{b}$, \eqref{eq:Q7} by  $\frac{iu^{\prime}}{2}$ and adding them together, we get
\begin{equation*}
\begin{split}
&(|b|^2-\frac{(u^\prime)^2}{4})q-i\lambda^2 a^2\bar{b}\bar{\psi}e^{-2u}s-i\lambda a^2 \bar{\psi}e^{-2u}\frac{iu^\prime}{2} t\\
&\quad\quad\quad \quad\quad\quad+(-\lambda |a|^2 a\bar{b}+\lambda^{-2}a^3 \frac{iu^{\prime}}{2})h=0.
\end{split}
\end{equation*}
Multiplying \eqref{eq:Q9} by $i\lambda a^2\bar{\psi}e^{-2u}$ and subtracting the above equation to eliminate $s$ and $t$, we conclude
$$q=\frac{\lambda^{-2} \frac{iu^{\prime}}{2}a^3-\lambda|a|^2 a\bar{b}-i\lambda a^3 \bar{\psi}e^{-u}}{
\frac{(u^\prime)^2}{4}-|b|^2+|\psi|^2 e^{-2u}}h.$$
Multiplying \eqref{eq:Q8} by $\frac{iu^{\prime}}{2}$, \eqref{eq:Q9} by $\lambda^{-1}b$, and adding them together yields
\begin{equation*}
\begin{split}
&-\lambda^{-1}b \frac{iu^\prime}{2}p-\lambda^{-2}b^2q+(\frac{(u^\prime)^2}{4}-|b|^2)s\\
&\quad\quad\quad \quad\quad\quad+(-\lambda^2\bar{a}e^u \frac{iu^\prime}{2}+ae^u \lambda^{-1}b)h=0.
\end{split}
\end{equation*}
Multiplying \eqref{eq:Q6} by $\lambda^{-1}b$ and adding the above equation to eliminate $p$ and $q$ results in
$$s=\frac{\lambda^2\bar{a}e^u \frac{iu^{\prime}}{2}-\lambda^{-1}abe^u+\lambda^{-1}|a|^2 ab}{\frac{(u^\prime)^2}{4}-|b|^2+|\psi|^2 e^{-2u}}h.$$

Finally, multiplying \eqref{eq:Q8} by $\lambda \bar{b}$, \eqref{eq:Q9} by $-\frac{iu^\prime}{2}$ and adding them together, we arrive at
\begin{equation*}
\begin{split}
&-|b|^2p+\lambda^{-1}b \frac{iu^\prime}{2} q+(|b|^2-\frac{(u^\prime)^2}{4})t
+(-\lambda^3\bar{a}\bar{b} e^u -ae^u \frac{iu^\prime}{2})h=0.
\end{split}
\end{equation*}
Multiplying \eqref{eq:Q7} by $\lambda^{-1}b$ and subtracting the above equation to eliminate $p$ and $q$, we obtain
$$t=\frac{-ae^u \frac{iu^{\prime}}{2}-\lambda^3\bar{a}\bar{b}e^u-\lambda^{-3}a^3b}{\frac{(u^\prime)^2}{4}-|b|^2+|\psi|^2 e^{-2u}}h.$$

Due to \eqref{eq:Q1}, the equation \eqref{eq:Q2} is equivalent to $-\Omega^{\prime}\gamma+A\xi=c\eta$. Moreover,  \eqref{eq:Q3} is equivalent to $-c\bar{\xi}^t=\zeta A$.
Substituting \eqref{eq:Q_assumption} into these two equations, we arrive at the following system of equations 
\begin{eqnarray}
&&\lambda^{-1}ap-\lambda \bar{a} q+(\lambda^{-1}\frac{iu^\prime}{2}-i\lambda^2 \bar\psi e^{-u})a^2h=\lambda^{-1}ac,\label{eq:Q10}\\
&&\lambda^{-1} a s-\lambda \bar{a}t-(\lambda^{-2}a^2b+\lambda \frac{i u^{\prime}}{2} e^u)h=-\lambda a^{-1} e^u c, \label{eq:Q11}\\
&&-\lambda a^{-1} e^u p+\lambda^{-1} as =-\lambda \bar{a} c, \label{eq:Q12}\\
&&-\lambda a^{-1} e^u q+\lambda^{-1}a t =\lambda^{-1}ac. \label{eq:Q13}
\end{eqnarray}
Note, since $a \neq 0$, equation \eqref{eq:Q13} yields  $c=t-\lambda^2 a^{-2} e^u q$.
Substituting the expression for $q$ derived above, we obtain 
$$c=\frac{ia(\lambda^3 \bar\psi-\lambda^{-3}\psi-e^u u^{\prime})}{\frac{(u^\prime)^2}{4}-|b|^2+|\psi|^2 e^{-2u}}h.$$ 
By a  direct computation, we see that \eqref{eq:Q12} is an identity and \eqref{eq:Q10} and \eqref{eq:Q11} are both equivalent to \eqref{eq:equiv1}.

In view of  equations \eqref{eq:equiv1} and \eqref{eq:equiv2}, we obtain
\begin{equation}\label{eq:checkQ}
\hat{Q}=\frac{iah}{2(|a|^2-e^u)} \check{Q}, \quad 
\check{Q}=\begin{pmatrix}
\check{p}&\check{q}&\check{v}_1\\ 
\check{s}&\check{t}&\check{v}_2\\
0&0&\check{c}
\end{pmatrix},
\end{equation}
where
\begin{equation}\label{eq:pqstvc}
\begin{split}
\check{p}&= -|a|^2 \frac{u^{\prime}}{2}+\lambda^3 \bar{\psi}|a|^2 e^{-u}-\lambda^{-3}\psi,\\
\check{q}&= \frac{\lambda^{-2}a}{\bar{a}}[\frac{u^{\prime}}{2} |a|^2 -\lambda^3\bar{\psi}e^{-u}(|a|^2-e^u)],\\
\check{s}&=\frac{\lambda^2}{a^2}[|a|^2 \frac{u^{\prime}}{2}e^u+\lambda^{-3} \psi (|a|^2- e^u)],\\
\check{t}&=\frac{1}{|a|^2} (-|a|^2 \frac{u^{\prime}}{2}e^u+\lambda^3 \bar{\psi}e^u -\lambda^{-3}\psi |a|^2),\\
\check{v}_1&=-2i\lambda^{-1}a (|a|^2-e^u),\\
\check{v}_2&=-2i \lambda a^{-1} e^u (|a|^2-e^u),\\
\check{c}&=\lambda^3 \bar{\psi}-\lambda^{-3}\psi-e^u u^{\prime}.
\end{split}
\end{equation}

So far we did not impose any restrictions on $\hat{Q}$. In particular, we ignored possible poles in $\lambda$ and in $z$.
 It is easy to verify that all matrix entries of $\hat{Q}$ are defined for sufficiently small $\lambda\in \C^*$.
In addition, we would like to impose now the condition for $\hat{Q}$ to have determinant $1$.
Computing this determinant we obtain
\begin{equation*}\label{eq:det}
(\check{p}\check{t}-\check{q}\check{s})\check{c} = (\lambda^3 \bar{\psi}-\lambda^{-3}\psi
- e^u u^{\prime})^2( \lambda^3 \bar{\psi}-\lambda^{-3}\psi)\in \mathbb{C}.
\end{equation*}

For small $\lambda\in \mathbb{C}^*$, 
we define
\begin{equation} \label{eq:Q}
\tilde{Q}=\frac{\lambda^3}{\kappa} \check{Q},
\end{equation}
where $\kappa=(\lambda^6 \bar{\psi}-\psi -\lambda^3e^u u^{\prime})^{2/3}( \lambda^6 \bar{\psi}-\psi)^{1/3}$.
Then $\det \tilde{Q}=1$
and   $\tilde{Q}(0,\lambda)=Q_0(0,\lambda)=I$ due to $a=ie^{\frac{u(0)}{2}}$.
Moreover, $\tilde{Q}$ is holomorphic in $\lambda$ in a small disk about $\lambda = 0$.
If $\lambda$ is small, the denominator of the coefficient of $\tilde{Q}$ single-valued.
Altogether we have found a solution to 
equation \eqref{eq:QDQ-1} by $Q=Q_0 \tilde{Q}$.

\subsection{Solving equation \eqref{eq:Equ2}}
\label{subsect:solve-2}

Since also $U_+$ has the same properties as $Q$, we obtain that $E = Q^{-1} U_+$ 
has determinant $1$, attains the value $I$ for $z=0$, is holomorphic for all small $\lambda$ and satisfies
$[Q^{-1}U_{+}, D]=0$.

By Remark \ref{rem:six lambda}
we can assume without loss of generality that $D=D(\lambda )$  is regular semi-simple for all but finitely many values of $\lambda$. Therefore, for all 
$z$ and small $\lambda$ we can write
$E = \exp ( \mathcal{E} ) $, where $[\mathcal{E}, D]=0.$

Since, in the computation of $Q$, we did not worry about the twisting condition, the matrix $\mathcal{E}$ is possibly an untwisted loop matrix in $SL(3,\C)$.  But since $SL(3,\C)$ has rank $2$, for any regular semi-simple matrix $D=D(\lambda),$ the commutant of $D(\lambda)$ is spanned by $D(\lambda)$ and one other matrix.

\begin{lemma} Every element in the commutant 
$\{X\in \Lambda sl(3,\mathbb{C})_{\sigma}: [X,D]=0\}$
of $D(\lambda)$ has the form 
$X(\lambda)=\kappa_1(\lambda) D(\lambda)+\kappa_2(\lambda)L_0(\lambda)$ with 
$\kappa_1(\epsilon\lambda)=\kappa_1(\lambda)$, $\kappa_2(\epsilon\lambda)=-\kappa_2 (\lambda)$,
where $L_0=D^2(\lambda) -\frac{1}{3} \mathrm{tr}(D^2)  I$.
\end{lemma}

Hence,  the matrix $Q^{-1} U_+$ has the form
\begin{equation*}
Q^{-1}U_{+}=\exp(\beta_1 D+\beta_2 L_0),
\end{equation*}
where $\beta_1$ and $\beta_2$ are functions of $y$ and $\lambda$ near $0$.
Thus equation \eqref{eq:Equ2} leads to
\begin{equation*}\label{eq:beta12}
\beta_1^{\prime}D+\beta_2^{\prime}L_0=-Q^{-1}\frac{d}{dy}Q+2i Q^{-1}(V_0+\lambda V_1)Q.
\end{equation*}
Recalling  $Q=Q_0 \tilde{Q}$, we obtain
\begin{equation}\label{eq:beta12new}
\beta_1^\prime \tilde{Q}D+\beta_2^{\prime} \tilde{Q}L_0=(-Q_0^{-1} \frac{dQ_0}{dy}+2iQ_0^{-1}VQ_0)\tilde{Q}-\frac{d\tilde{Q}}{dy}.
\end{equation}

A direct computation shows
$$\tilde{Q}D=\frac{\lambda^3}{\kappa}\begin{pmatrix}
\lambda^{-1}b\check{q}-\lambda \bar{a}\check{v}_1 & -\lambda\bar{b}\check{p}+\lambda^{-1}a\check{v}_1&\lambda^{-1}a\check{p}-\lambda\bar{a}\check{q}\\
\lambda^{-1}b\check{t}-\lambda \bar{a}\check{v}_2 & -\lambda\bar{b}\check{s}+\lambda^{-1}a\check{v}_2&\lambda^{-1}a\check{s}-\lambda\bar{a}\check{t}\\
-\lambda \bar{a}\check{c}&\lambda^{-1}a \check{c}&0
\end{pmatrix},$$
$$L_0=\begin{pmatrix}
\frac{|a|^2-|b|^2}{3}&\lambda^{-2}a^2&\lambda^2 \bar{a}\bar{b}\\
\lambda^2\bar{a}^2&\frac{|a|^2-|b|^2}{3}&\lambda^{-2}ab\\
\lambda^{-2}ab&\lambda^2\bar{a}\bar{b}&-\frac{2}{3}(|a|^2-|b|^2)
\end{pmatrix},$$
where
$$\quad \lambda^3 L_0 \in \Lambda SL(3,\mathbb{C})_{\sigma},
 \hspace{2mm} \text{and} \hspace{2mm} \hat{\sigma}(L_0(\lambda)):=\sigma(L_0(\varepsilon^{-1}\lambda))=-L_0(\lambda)$$
and 
\begin{eqnarray*}
&&\tilde{Q}L_0=\frac{\lambda^3}{\kappa}\cdot 
\begin{pmatrix}
\frac{|a|^2-|b|^2}{3}\check{p}+\lambda^2 \bar{a}^2\check{q}+\lambda^{-2}ab \check{v}_1&
*
&*\\
\frac{|a|^2-|b|^2}{3}\check{s}+\lambda^2 \bar{a}^2\check{t}+\lambda^{-2}ab\check{v}_2&* &*\\
\lambda^{-2}ab\check{c}& \lambda^2\bar{a}\bar{b}\check{c}&-\frac{2}{3}(|a|^2-|b|^2)\check{c}
\end{pmatrix}.
\end{eqnarray*}

On the other hand, 
$$-Q_0^{-1}\frac{dQ_0}{dy}+2iQ_0 VQ_0=\begin{pmatrix}
0&-2\lambda \bar{\psi}a^2 e^{-2u}&0\\
0&0&-\frac{2i\lambda e^u}{a}\\
-\frac{2i\lambda e^u}{a}&0&0
\end{pmatrix},$$

and

$$(-Q_0^{-1}\frac{dQ_0}{dy}+2iQ_0 VQ_0)\tilde{Q}=
\frac{\lambda^3}{\kappa}
\begin{pmatrix}
-2\lambda \bar{\psi}a^2 e^{-2u}\check{s}&-2\lambda \bar{\psi} a^2 e^{-2u}\check{t}&-2\lambda\bar{\psi}a^2e^{-2u}\check{v}_2\\
0&0&-\frac{2i\lambda e^u}{a}\check{c}\\
-\frac{2i\lambda e^u}{a}\check{p}&-\frac{2i\lambda e^u}{a}\check{q}&-\frac{2i\lambda e^u}{a}\check{v}_1
\end{pmatrix},
$$
and 
$$\frac{d\tilde{Q}}{dy}=\frac{\lambda^3}{\kappa} (\check{Q}^{\prime}+\frac{\frac{2}{3} \lambda^3 e^u[(u^{\prime})^2+u^{\prime\prime}]}
{\lambda^6\bar{\psi}-\psi-\lambda^3 e^u u^{\prime}} \check{Q}).$$

Substituting this into \eqref{eq:beta12new}  we obtain  $9$ equations for $\beta_1^{\prime}$ and $\beta_2^{\prime}$. 
In particular, 
we obtain
\begin{eqnarray}
&-\beta_1^{\prime}\lambda \bar{a} +\beta_2^{\prime}\lambda^{-2}ab=-\frac{2i\lambda e^u}{a}\frac{\check{p}}{\check{c}}\label{eq:31}\\
&\beta_1^{\prime} \lambda^{-1}a+\beta_2^{\prime}\lambda^2 \bar{a}\bar{b}=-\frac{2i\lambda e^u}{a}\frac{\check{q}}{\check{c}}\label{eq:32}.
\end{eqnarray}
Solving \eqref{eq:31} and \eqref{eq:32} and integrating yields
\begin{equation}\label{eq:equiv_beta}
\begin{split}
\beta_1(y)&= \int_0^y \frac{2i \lambda^3\bar{\psi} - iu^{\prime} e^u}{\lambda^3\bar{\psi}-\lambda^{-3}\psi-e^u u^{\prime}} ds, \\
\beta_2(y) &= \int_0^y \frac{2e^u}{\lambda^3\bar{\psi}-\lambda^{-3}\psi-e^u u^{\prime}}ds. 
\end{split}
\end{equation}
Since we already assume $u^{\prime}\neq 0$,
we find that the solutions $\beta_1^{\prime}$ and $\beta_2^{\prime}$ also satisfy the other $7$ equations. 

Putting everything together we obtain

\begin{theorem} [Explicit Iwasawa decomposition]
The extended frame $\F$, satisfying $\F(0,\lambda) = I$, for the translationally equivariant minimal Lagrangian surface in $\mathbb{C}P^2$ generated by the  potential $D(\lambda)dz$  with vanishing diagonal, and satisfying $ab \neq 0,$ is given by
\begin{equation*}\label{eq:extended frame F}
\mathbb{F}(z,\lambda)=\exp(zD-\beta_1(y,\lambda) D-\beta_2(y,\lambda) L_0) Q^{-1}(y,\lambda),
\end{equation*}
with $\beta_1, \beta_2$ as in \eqref{eq:equiv_beta} and $Q=Q_0\tilde{Q}$ as in \eqref{eq:Q_0}, \eqref{eq:checkQ}, \eqref{eq:pqstvc}, \eqref{eq:Q} and $u$ as in  \eqref{sol}.
\end{theorem}

\begin{remark}
In the proof of the last theorem we have derived the equation
 $U_{+}=Q\exp(\beta_1 D+\beta_2 L_0).$
In this equation each separate term is only defined for small $\lambda$ and a possibly restricted set of $y'$s. However, due to the globality and the uniqueness of the Iwasawa splitting, the matrix $U_+$ is defined for all $\lambda$ in $\C$ and all $z \in \C$.
\end{remark}

\subsection{Explicit expressions for minimal Lagrangian immersions}

We know from Remark \ref{rem:six lambda} that 
except for special cases (which we have excluded)  the matrix $D$ has three different eigenvalues. 
Since $D$ is skew-Hermtian, the corresponding eigenvectors are automatically perpendicular. Therefore there exists a unitary matrix $L$ such that 
$D=L\mathrm{diag}(i d_1,id_2, id_3) L^{-1}$.
As a consequence, for the extended lift $F$ we thus obtain
$$ F = \F e_3=L\exp(z\Lambda-\beta_1\Lambda-\beta_2(\Lambda^2-\frac{\mathrm{tr}\Lambda^2}{3}I))L^{-1}Q^{-1}e_3,$$
where $\Lambda=\mathrm{diag}(i d_1, i d_2, id_3)$. 
Set $ L= ( l_1,l_2,l_3)$. 
Altogether we have shown
\begin{theorem}[\cite{DorfMa1}] Every translationally equivariant minimal Lagrangian immersion generated by the potential $D(\lambda) dz$ has a canonical lift 
$F = F(z,\lambda)$ of the form
\begin{equation}\label{F-loop}
\begin{split}
F(z,\lambda)=&
\sum_{j=1}^3 \exp\{ i z d_j (\lambda) - i \beta_1(y, \lambda) d_j(\lambda) + \beta_2(y,\lambda))
 (d_j(\lambda)^2 - \frac{2\beta}{3} )\}\\
 &\quad\quad\quad
\langle Q^{-1}e_3, l_j\rangle l_j.
\end{split}
\end{equation}
\end{theorem}

Along the ideas of the paper \cite{CU94} by Castro-Urbano we have obtained in \cite{DorfMa1} 
\begin{theorem}
\begin{enumerate}
\item When the cubic differential $\lambda^{-3}\Psi$ of an translationally equivariant minimal Lagrangian immersion $f$ is not real, the canonical lift  $F$ of $f$ has  the form 
\begin{equation}\label{eq:F(x,y,lambda)}
F(x,y, \lambda)= \sum_{j=1}^3 h_j(y)e^{id_j x +iG_j(y)} \hat{l}_j,
\end{equation}
where 
\begin{equation}\label{eq:h_j G_j}
h_j(y)=\left(\frac{d_je^u-\mathrm{Re}(\lambda^{-3}\psi)}{d_j^3-\mathrm{Re}(\lambda^{-3}\psi)}\right)^{\frac{1}{2}},
\quad\quad
G_j(y)=\int_0^y \frac{d_j \mathrm{Im}(\lambda^{-3}\psi)}{d_j e^u-\mathrm{Re}(\lambda^{-3}\psi)}ds.
\end{equation}
\item When $\lambda^{-3}\Psi$ is real, the canonical lift $F$ of $f$ has the form 
\begin{equation*}\label{eq:F-psi real}
F(x,y,\lambda)=\sum_{j=1}^3 \epsilon_j (\frac{\beta}{3}-e_j) \sqrt{\frac{e_j-\wp(y-\omega^{\prime})}{8|\psi|^2-(\frac{\beta}{3}-e_j)^2}} e^{id_j x} \hat{l}_j
\end{equation*} 
with $\epsilon_1=-1, \epsilon_2=\epsilon_3=1$.

\end{enumerate}
\end{theorem}

Since both $l_j$ and $\hat{l}_j$ are orthonormal eigenvectors of $D(\lambda)$ with respect to the eigenvalue $id_j$ and are independent of $z$, there exists a phase factor $\alpha_j$ which is also independent of $z$ such that $l_j = \hat{l}_j  \alpha_j$. Then we can check straightforwardly at the point $y=2\omega$ that the two horizontal lifts  \eqref{F-loop}, obtained by using the loop group method  and \eqref{eq:F(x,y,lambda)} obtained by using the idea of Castro-Urbano, are the same up to a constant factor of length $1$.

We can also prove the following theorem directly, which appears as Theorem 7.1 in \cite{DorfMa1}.

\begin{theorem}
For every translation $z\mapsto z+p + i m2\omega$  for $p\in \R$, $m\in \Z$, and $j = 1,2,3$,  the equation
\begin{equation*}
G_j(2\omega,\lambda)+\mathrm{Re}\beta_1(2\omega, \lambda) d_j(\lambda)
+\mathrm{Im}\beta_2(2\omega,\lambda)(-d_j(\lambda)^2+\frac{2\beta}{3}) = 0.
\end{equation*}
holds.
\end{theorem}

\begin{proof}
Since 
\begin{equation}\label{eq:beta1}
\begin{split}
&\beta_1(y)=i\int_0^y \frac{e^u u^{\prime}-2\lambda^3 \bar{\psi}}{e^u u^{\prime}+2i \mathrm{Im}(\lambda^{-3}\psi)} ds\\
&=i \int_0^y \frac{(e^u u^{\prime})^2-2\mathrm{Re}(\lambda^{-3}\psi)e^{u}u^{\prime}+4i\lambda^3\bar{\psi}\mathrm{Im}(\lambda^{-3}\bar{\psi})}{(e^u u^{\prime} )^2+4[\mathrm{Im}(\lambda^{-3}\psi)]^2}ds\\
&=iy-2i\mathrm{Re}(\lambda^{-3}\psi)\int_0^y \frac{e^u u^{\prime}}{(e^u u^{\prime})^2+4[\mathrm{Im}(\lambda^{-3}\psi)]^2}ds\\
& \quad\quad\quad 
-4\mathrm{Im}(\lambda^{-3}\psi)\mathrm{Re}(\lambda^{-3}\psi) \int_0^y \frac{ds}{(e^u u^{\prime})^2+4[\mathrm{Im}(\lambda^{-3}\psi)]^2},
\end{split}
\end{equation}
and
\begin{equation}\label{eq:beta2}
\begin{split}
&\beta_2(y)=-\int_0^y \frac{2 e^u}{e^u u^{\prime}+2i \mathrm{Im}(\lambda^{-3}\psi)} ds\\
&=-2\int_0^y \frac{e^{2u} u^{\prime}}{(e^u u^{\prime} )^2+4[\mathrm{Im}(\lambda^{-3}\psi)]^2}ds+4i\mathrm{Im}(\lambda^{-3}\psi)\int_0^y\frac{e^u}{(e^u u^{\prime})^2 + 4[\mathrm{Im}(\lambda^{-3}\psi)]^2}ds,
\end{split}
\end{equation}
thus we have 
\begin{eqnarray*}
\mathrm{Re}\beta_1(y)&=&-4\mathrm{Im}(\lambda^{-3}\psi)\mathrm{Re}(\lambda^{-3}\psi) \int_0^y \frac{ds}{(e^u u^{\prime})^2+4[\mathrm{Im}(\lambda^{-3}\psi)]^2} ds,\\
\mathrm{Im}\beta_2(y)&=&4\mathrm{Im}(\lambda^{-3}\psi)\int_0^y\frac{e^u}{(e^u u^{\prime})^2 + 4[\mathrm{Im}(\lambda^{-3}\psi)]^2}ds.
\end{eqnarray*}
Hence,
\begin{eqnarray*}
&&G_1(y)+\mathrm{Re}\beta_1(y) d_1+\mathrm{Im}\beta_2(y)(-d_1^2+\frac{2\beta}{3})\\
&=&\mathrm{Im}(\lambda^{-3}\psi)\int_0^y\{ \frac{d_1}{d_1 e^u - \mathrm{Re}(\lambda^{-3}\psi)}
+\frac{-4d_1\mathrm{Re}(\lambda^{-3}\psi)+4(-d_1^2+\frac{2\beta}{3})e^u}{(e^u u^{\prime})^2+4[\mathrm{Im}(\lambda^{-3}\psi)]^2}\}ds.
\end{eqnarray*}
Denote $w(y)=e^{u(y)}$ as before. Recall that
\begin{equation*}
\begin{split}
&d_j^3-\beta d_j+2\mathrm{Re}(\lambda^{-3}\psi)=0,\\
&w^{\prime\prime}+12 w^2-4\beta w=0, \quad \text{ if } w^{\prime}\neq 0.
\end{split}
\end{equation*}
As a consequence, we have
\begin{equation*}
\begin{split}
&d_1 \{ (w^{\prime})^2+4 [\mathrm{Im}(\lambda^{-3}\psi)]^2\}+[-4d_1\mathrm{Re}(\lambda^{-3}\psi)+4(-d_1^2+\frac{2\beta}{3})w][d_1 w-\mathrm{Re}(\lambda^{-3}\psi)]\\
&=-8d_1 w^3 +4 (-d_1^3 +\frac{5\beta}{3}d_1)w^2-\frac{8}{3}\beta \mathrm{Re}(\lambda^{-3}\psi)w\\
&=\frac{2}{3}[ d_1 w-\mathrm{Re}(\lambda^{-3}\psi)] w^{\prime\prime}.
\end{split}
\end{equation*}
Therefore,
\begin{equation*}
\begin{split}
&G_1(y)+\mathrm{Re} \beta_1(y) d_1 +\mathrm{Im}\beta_2(y) (-d_1^2+\frac{2}{3})\\
&=\frac{2}{3}\mathrm{Im}(\lambda^{-3}\psi)\int_0^y\frac{w^{\prime\prime}}{(w^{\prime})^2+4[\mathrm{Im}(\lambda^{-3}\psi)]^2}ds\\
&=\frac{1}{3}\arctan \frac{w^{\prime}(y)}{2\mathrm{Im}(\lambda^{-3}\psi)},
\end{split}
\end{equation*}
and $$G_1(2\omega)+\mathrm{Re} \beta_1(2\omega) d_1 +\mathrm{Im}\beta_2(2\omega) (-d_1^2+\frac{2}{3})=0.$$
 The proof for the cases $j=2$ and $3$ is analogous.
\end{proof}

\section{Explicit expressions for $\beta_j$ and $G_j$ in terms of Weierstrass elliptic functions}
\label{Sec:quantities}

In this section, we present completely explicit expressions for the quantities $\beta_1(2\omega) $, $\beta_2(2\omega) $, and $G_j(2 \omega)$.

From \eqref{eq:beta1}  we obtain
\begin{equation}
\begin{split}
&\beta_1(2\omega)=2\omega i-16\mathrm{Im}(\lambda^{-3}\psi)\mathrm{Re}(\lambda^{-3}\psi) 
\int_0^{2\omega} \frac{dy}{[\wp^{\prime}(y-\omega^{\prime})]^2+16[\mathrm{Im}(\lambda^{-3}\psi)]^2},\\
&=2\omega i-16\mathrm{Im}(\lambda^{-3}\psi)\mathrm{Re}(\lambda^{-3}\psi) \int_0^{2\omega}\frac{dy}{4[\wp(y-\omega^{\prime})]^3-g_2 \wp(y-\omega^{\prime})-\tilde{g}_3},\label{eq:beta1=}
\end{split}
\end{equation}
where $\tilde{g}_3=g_3-16[\mathrm{Im}(\lambda^{-3}\psi)]^2$.

It is easy to derive from \eqref{Delta-e} that the cubic equation $4\wp^3-g_2 \wp-\tilde{g}_3=0$ has three distinct roots which we will  order to satisfy
$\tilde{e}_1>\tilde{e}_2>\tilde{e}_3$. Thus
\begin{equation}\label{eq:p-e}
\begin{split}
&\int_0^{2\omega} \frac{dy}{\wp(y-\omega^{\prime})^3-\frac{g_2}{4} \wp(y-\omega^{\prime})- \frac{\tilde{g}_3}{4}}\\
=&\int_0^{2\omega} [\frac{1}{\tilde{H}_1^2} \frac{1}{\wp(y-\omega^{\prime})-\tilde{e}_1}
+\frac{1}{\tilde{H}_2^2} \frac{1}{\wp(y-\omega^{\prime})-\tilde{e}_2}
+\frac{1}{\tilde{H}_3^2} \frac{1}{\wp(y-\omega^{\prime})-\tilde{e}_3}]dy,
\end{split}
\end{equation}
where 
\begin{equation}\label{H_i}
\tilde{H}_i^2:=(\tilde{e}_i-\tilde{e}_j)(\tilde{e}_i-\tilde{e}_k)
\end{equation}
for distinct $i,j,k\in\{1,2,3\}$.

If $\lambda^{-3}\psi$ is not real, then $\tilde{g}_3\neq g_3$. Denote $\wp(\alpha_j)=\tilde{e}_j$ for $j=1,2$ or $3$.
By using formula  (18.7.3)  of \cite{AS}
$$\int \frac{du}{\wp(u)-\wp(\alpha)}=\frac{1}{\wp^{\prime}(\alpha)}[\ln\frac{\sigma(u-\alpha_1)}{\sigma(u+\alpha)}+2u\zeta(\alpha)]$$
for $\wp(\alpha)\neq e_1, e_2$ or $e_3$,
and (18.2.20)
$$\sigma(z+2m\omega+2n\omega^{\prime})=(-1)^{m+n+mn}\sigma(z)\exp[(z+m\omega+n\omega^{\prime})(2m\eta+2n\eta^{\prime})],$$
where $\eta=\zeta(\omega)$, $\eta^{\prime}=\zeta(\omega^{\prime})$,
we obtain
\begin{equation}\label{eq:p-ej}
\int_0^{2\omega} \frac{dy}{\wp(y-\omega^{\prime})-\tilde{e}_j}
=\frac{1}{\wp^{\prime}(\alpha_j)}[4\omega \zeta(\alpha_j)-4\alpha_j \eta].
\end{equation} 
Here $\sigma(z)=\sigma(z;g_2,g_3)$ and $\zeta(z)=\zeta(z;g_2,g_3)$ are the Weierstrass $\sigma$-function and the Weierstrass $\zeta$-function, respectively. Therefore, substituting \eqref{eq:p-e}, \eqref{H_i} and \eqref{eq:p-ej} into \eqref{eq:beta1=} we derive
\begin{equation*}
\beta_1(2\omega)=2\omega i-4\mathrm{Im}(\lambda^{-3}\psi)\mathrm{Re}(\lambda^{-3}\psi)\sum_{j=1}^3\frac{1}{\tilde{H}_j^2}\frac{1}{\wp^{\prime}(\alpha_j)} [4\omega\zeta(\alpha_j)-4\alpha_j \eta]. 
\end{equation*}

Similarly,  from \eqref{eq:beta2} and \eqref{eq:p-ej}  we derive 
\begin{eqnarray*}
&&\beta_2(2\omega)=4i\mathrm{Im}(\lambda^{-3}\psi)\int_0^{2\omega}
\frac{\frac{\beta}{6}-\frac{\wp(y-\omega^{\prime})}{2}}{4[\wp(y-\omega^{\prime})]^3-g_2 \wp(y-\omega^{\prime})-\tilde{g}_3}dy,\\
&=&i\mathrm{Im}(\lambda^{-3}\psi) \sum_{j-1}^{3} \int_0^{2\omega}
 \frac{\frac{\beta}{3}-\tilde{e}_j}{2\tilde{H}_j^2}\frac{1}{\wp(y-\omega^{\prime})-e_j} dy\\
 &=&i\mathrm{Im}(\lambda^{-3}\psi) \sum_{j-1}^{3} \frac{\frac{\beta}{3}-\tilde{e}_j}{2\tilde{H}_j^2} \frac{1}{\wp^{\prime}(\alpha_j)}
 [4\omega \zeta(\alpha_j)-4\alpha_j \eta].
\end{eqnarray*}

Notice that actually $\frac{\beta}{3}+d_2d_3=\tilde{e}_1$ holds. From \eqref{eq:h_j G_j} and \eqref{eq:p-ej} we compute
\begin{eqnarray*}
G_1(2\omega)&=&2\mathrm{Im}(\lambda^{-3}\psi)\int_0^{\omega} \frac{1}{2e^{u(y)}+d_2d_3}dy\\
&=&-2\mathrm{Im}(\lambda^{-3}\psi) \int_0^{\omega} \frac{dy}{\wp(y-\omega^{\prime})-(\frac{\beta}{3}+d_2d_3)}\\
&=&-2\mathrm{Im}(\lambda^{-3}\psi) \frac{4\omega \zeta(\alpha_1)-4\alpha_1 \eta}{\wp^{\prime}(\alpha_1)}.
\end{eqnarray*} 
Similarly, we obtain
$$G_j(2\omega)=-2\mathrm{Im}(\lambda^{-3}\psi) \frac{4\omega \zeta(\alpha_j)-4\alpha_j \eta}{\wp^{\prime}(\alpha_j)}$$
for $j=2, 3$.

When $\lambda^{-3}\psi$ is real, it is obvious that $\beta_1(2\omega)=2\omega i$, $\beta_2(2\omega)=0$ and 
$G_j(2\omega)=0$ holds.

\section{Homogeneous minimal Lagrangian immersions into $\C P^2$}
\label{Sec:Hom}

In this section we will consider homogeneous minimal Lagrangian immersions into 
$\C P^2$ from the point of view of the loop group method.

\subsection{Basic results}

In our context we consider minimal Lagrangian immersions
for which the group of symmetries acts transitively.

More precisely, we consider minimal Lagrangian immersions 
$f:M \rightarrow \C P^2$, where $M$ is a Riemann surface,  and consider their group of symmetries
\begin{small}
\begin{equation*} \label{symmetrygroup}
\Gamma^f_\mathcal{S} = \lbrace (\gamma, R) \in  Aut(M) \times SU(3),
f(\gamma\cdot z) = R\cdot f(z) \hspace{2mm} \mbox{for all} \hspace{2mm}  z \in M \rbrace.
\end{equation*} 
\end{small}

We will also consider its group $\Gamma^f_M$ of projections onto the first component
\begin{equation*}
\begin{split}
\Gamma^f_M& = \lbrace \gamma \in Aut(M);  \mbox{there exists some } R \in SU(3)  \mbox{ such that } \\
&\hspace{16mm}  (\gamma,R) \in \Gamma^f_\mathcal{S} \rbrace.
\end{split}
\end{equation*}

It is easy to see that the following statements hold

\begin{lemma}
The group $\Gamma^f_\mathcal{S}$ is closed in the Lie group 
$Aut(M) \times SU(3)$ and the group $\Gamma^f_M $ is closed in $Aut(M)$.
In particular, both groups are Lie groups.
\end{lemma}

\begin{definition}
A minimal Lagrangian immersion $f:M \rightarrow \C P^2$ is homogeneous,
if the Lie group $\Gamma^f_M $ acts transitively in $M$.
\end{definition}

\begin{remark} 
(1) If a minimal Lagrangian immersion $f:M \rightarrow \C P^2$ is homogeneous, then it is clear that also its lift to the universal cover
$\tilde{M} $ of $M$ is homogeneous.

(2) We can replace, without loss of generality, all groups considered so far by their connected components containing the identity element $I$.

(3) Since homogeneity implies that the groups under consideration are Lie groups, a homogeneous minimal Lagrangian immersion is in particular equivariant.
\end{remark}

\subsection{Homogeneous minimal Lagrangian surfaces defined on 
simply-connected  Riemann surfaces}

As pointed out in $(3)$ of the remark above, the notion of \lq\lq homogeneous\rq\rq \, implies \lq\lq equivariant\rq\rq, but, obviously, is much stronger. This will show up clearly in the discussion below.

Since there are exactly three simply connected Riemann surfaces, namely the Riemann sphere $S^2$, the unit disk $\D$, and the complex plane $\C$, we will separate the discussion accordingly.

\subsubsection{ The case $ M = S^2$} 

The case of $S^2$ is very special, since then  the cubic differential $\psi$ on $S^2$ vanishes identically. As pointed out in \cite{DorfMa0}, this implies that the normalized potential of such an immersion is nilpotent and only depends on one function.
 As a consequence, the normalized potential can be assumed (almost everywhere, at least locally after some change of coordinates) to be constant, and, since $\psi =0$, to be nilpotent. Therefore, as stated explicitly in \cite{DorfMa0}, one can carry out the loop group method explicitly and obtains that the image of $f$ is contained in some isometric image of $\R P^2$.
The same result has been obtained before by classical differential geometric methods in \cite{EschenburgGT, Wang}. 

We thus obtain:

\begin{theorem}
Every minimal Lagrangian immersion $f:S^2 \rightarrow \C P^2$ is homogeneous and $f(S^2)$ is, up to isometries of $\C P^2$ contained in 
$\R P^2$.
\end{theorem}

\subsubsection{ The case $ M = \D$} 

In this case the group $\Gamma^f_M $ is a closed subgroup of  $Aut(\D) = SL(2, \R)$ which acts transitively on $\D$. Now it is easy to see, by performing a Levi decomposition of $\Gamma^f_M $ and then an Iwasawa decomposition of its semi-simple part that $\Gamma^f_M $ actually contains (up to conjugation) the group of upper triangular matrices $\Delta$ in $SL(2,\R)$  for which the diagonal elements are positive. But the image of  the connected solvable group $\Delta$ in $SU(3)$ under the monodromy representation needs to be unitary.
Therefore, the kernel of the monodromy representation has at least dimension $1$ and the image of $f$ would be one-dimensional.

\begin{theorem}
There does not exist any  homogeneous, minimal Lagrangian immersion  $f: \D \rightarrow \C P^2$ .
\end{theorem}

\subsubsection{ The case $M = \C$}
This case is slightly more complicated. The basic result is

\begin{theorem}
If $f:\C \rightarrow \C P^2$ is a homogeneous minimal Lagrangian immersion, the the group  $\Gamma^f_M $ contains, up to conjugation,  the subgroup of all translations of $\C$.
\end{theorem}
\begin{proof}
We note that any element $X$ of the Lie algebra of $Aut(\C)$ can be represented  in the form
\begin{equation*}
X = \begin{pmatrix}
u& v\\
0&0
\end{pmatrix}
\end{equation*}

Choosing the base point $z = 0$, then the transitivity of the action means that all $v \in \C$ will occur (with a certain $u = u(v) $).

If there is a transitive subgroup such that all elements of its Lie algebra have
$u=0$, then this subalgebra consists of translations only and the claim follows.

Assume now there exists some $X$ with $u \neq 0$. Then it is straightforward to see that there exists some $z_0 \in \C$ such that $\exp( tX)\cdot z_0 = z_0$ for all $t \in \R$, where $z_0$ does not depend on $t$.

Hence, after a change of the base point we can assume that $X$ is of the form $v =0$. But then, it is clear, that the commutators of such an $X$ with an arbitrary $Y$ of the Lie algebra of some transitive group under consideration form a two-dimensional Lie algbera and only consist of translations. Hence we obtain an abelian, transitive group of translations. 
\end{proof}

Now we obtain (also see \cite{EschenburgGT, Wang})

\begin{theorem}
Every homogeneous minimal Lagrangian immersion $f:\C \rightarrow \C P^2$ 
is isometrically isomorphic with the Clifford torus.
\end{theorem}
\begin{proof}
Every homogeneous minimal Lagrangian immersion is a doubly equivariant immersion. Therefore the metric and the cubic differential both are constant. 
As a consequence, the Maurer-Cartan form of the frame $\F$ is constant.
Writing $\F^{-1} d \F = Xdz + \tau(X) d\bar{z}$, we observe that the 
integrability condition implies $ \lbrack X, \tau(X) \rbrack = 0$. Since $X$ is of the form $X = \lambda^{-1} X_{-1} + X_0^\prime$  and $\F(z, \bar{z}, \lambda) = \exp(z X) \exp( \bar{z} \tau(X))$ we see that the immersion is generated by the potential $X$. A straightforward computation shows  that
$ \lbrack X, \tau(X) \rbrack = 0$, which implies $X_0 = 0$. Therefore the potential
$X$ is a vacuum, whence the immersion is isometrically isomorphic to the Clifford torus (see Section \ref{sec:vacuum}).
\end{proof}

\begin{remark}
The classification of all homogeneous minimal surfaces and   the classification of all homogeneous surfaces in $\C P^2$  have been  discussed  in \cite{EschenburgGT, Wang}, respectively. Here we just consider the classification of homogeneous minimal Lagrangian surfaces using the  loop group method.
For the surfaces under consideration the proof is quite simple and direct. 
\end{remark}

\begin{acknow}
The second author would like to express her sincere gratitude to Professor Robert Conte for his guidance on elliptic functions.
\end{acknow}


\end{document}